\documentclass[a4paper,12pt]{article}
\usepackage{graphicx}
\usepackage{setspace}
\usepackage{amsmath}
\usepackage{amssymb}
\usepackage{amsthm}
\newtheorem{thm}{Theorem}[section]
\theoremstyle{definition}

\theoremstyle{remark}

\theoremstyle{plain}
\newtheorem{lem}{Lemma}[section]

\author{\textbf{Zahid Raza and Anjum Iqbal}\\ Department of Mathematics\\
National University of Computer and Emerging Sciences,\\ Lahore Campus, Pakistan.\\ \texttt{zahid.raza@nu.edu.pk} \\ \texttt{anjum$\_$237@yahoo.com} }

\title{\textbf{Infinite Log-Concavity and r-Factor}}
\date{}
\begin{document}
\maketitle

\begin{abstract}

D. Uminsky and K. Yeats \cite{p1} studied the properties of the \emph{log-operator} $\mathcal{L}$ on the subset of the finite symmetric sequences and prove the  existence of  an infinite region $\mathcal{R} $,  bounded by parametrically defined hypersurfaces such that any sequence corresponding a point of $\mathcal{R}$ is \emph{infinitely log concave}.
We study the properties of a new operator $\mathcal{L}_r$ and redefine the hypersurfaces which generalizes the one defined by Uminsky and Yeats \cite{p1}. 

We show that any sequence corresponding a point of the region $\mathcal{R}$, bounded by the new generalized parametrically defined r-factor hypersurfaces, is \emph{Generalized r-factor infinitely log concave}. We also give an improved value of $r_\circ$ found by McNamara and Sagan \cite{p2} as the log-concavity criterion using the new log-operator.

\end{abstract}

\section{Introduction}

A sequence $(a_k)=a_0,a_1,a_2, \dots $ of real numbers is said to be \emph{log-concave}\ $\left( \mathrm{or}\ \emph{1-fold log-concave}\right)$ \textit{iff} the new sequence $ (b_k) $ defined by the $\mathcal{L}$ operator
$ (b_k)=\mathcal{L}(a_k)$ \ is non negative for all $ k \in N,$
where $\ b_k={a_k}^2 - a_{k-1} a_{k+1}$.
A sequence $(a_k)$ is said to be \emph{2-fold log-concave} \textit{iff}
$ \mathcal{L}^2(a_k)=\mathcal{L}(\mathcal{L}(a_k))
=\mathcal{L}(b_k)$ \ is non negative for all $ k \in N,$
where $\ \mathcal{L}(b_k)={b_k}^2 - b_{k-1} b_{k+1}$ and the sequence
$(a_k)$ is said to be \emph{i-fold log-concave} \textit{iff}
$\mathcal{L}^i(a_k)\ $is non negative for all  $k \in N,$
where \[\mathcal{L}^i(a_k)=[\mathcal{L}^{i-1}(a_k)]^2\  - [\mathcal{L}^{i-1}(a_{k-1})]\ [\mathcal{L}^{i-1}(a_{k+1})].\]
$(a_k)$ is said to be \emph{infinitely log-concave} \textit{iff}
$\mathcal{L}^i(a_k)\ $is non negative for all  $i\geq 1$.
Binomial coefficients ${n \choose 0} , {n \choose 1} , {n \choose 2} , \cdots $ \  along any row of Pascal's triangle are log concave for all $ n\geq0 $. Boros and Moll \cite{p3} conjectured that binomial coefficients along any row of Pascal's triangle are \emph{infinitely log-concave} for all $ n\geq0 $. This was later confirmed by P. McNamara and B. Sagan \cite{p2} for the $n^{\mbox{th}}$ rows of Pascal's triangle for $n\leq 1450$.

P. McNamara and B. Sagan \cite{p2} defined a stronger version of \emph{log-concavity}.\\
A sequence $(a_k)=a_0,a_1,a_2, \dots $ of real numbers is said to be \emph{r-factor log-concave} \emph{iff}
\begin{equation}
{a_k}^2 \geq r \ a_{k-1}\  a_{k+1}         \label{c11}
\end{equation}
for all \  $k\in N$.
Thus  \emph{r-factor log-concave} sequence implies \emph{log-concavity} if $r\geq 1$. We are interested only in \emph{log-concave} sequences, so from here onward, value of r used would mean $r\geq1$ unless otherwise stated.

We first define a new operator $\mathcal{L}_r$ and then using this operator, we define \emph{Generalized r-factor Infinite log-concavity} which is a bit more stronger version of \emph{log-concavity}.\\
Define the real operator $\mathcal{L}_r$ and the new sequence $(b_k)$ such that $(b_k)=\mathcal{L}_r(a_k)$,
where
\begin{align*}
b_k&={a_k}^2 -\ r\ a_{k-1}\  a_{k+1}  \label{c12}  \hspace{2.5in} \\
\mathrm{or} \hspace{1in} \mathcal{L}_r(a_k)&={a_k}^2 -\ r\  a_{k-1}\  a_{k+1}
\end{align*}
Then $(a_k)$ is said to be \emph{r-factor log-concave}\ (or \emph{Generalized r-factor 1-fold log-concave)\ iff}  $ (b_k) $\ is non negative for all $ k \in N.$ \\ This again defines $\left(\ref{c11}\right)$ alternatively using $\mathcal{L}_r$ operator.
$(a_k)$ is said to be \emph{Generalized r-factor 2-fold log-concave} \textit{iff}
$ \mathcal{L}_r^2(a_k)=\mathcal{L}_r(\mathcal{L}_r(a_k))
=\mathcal{L}_r(b_k)$ \ is non negative for all $ k \in N,$
where
\begin{align*}
{L}_r(b_k)&={b_k}^2 -\ r\ b_{k-1}\ b_{k+1} \hspace{2.5in}\\
\mathrm{or} \hspace{1in}
\mathcal{L}_r^2(a_k)&=[\mathcal{L}_r(a_k)]^2 -\ r\ [\mathcal{L}_r(a_{k-1})]\ [\mathcal{L}_r(a_{k+1})]
\end{align*}
$(a_k)$ is said to be \emph{Generalized r-factor i-fold log-concave iff}
$\mathcal{L}_r^i(a_k)\ $is non negative for all  $k \in N,$
where \[\mathcal{L}_r^i(a_k)=[\mathcal{L}_r^{i-1}(a_k)]^2\  - \ r\ [\mathcal{L}_r^{i-1}(a_{k-1})]\ [\mathcal{L}_r^{i-1}(a_{k+1})]\]
$(a_k)$ is said to be \emph{Generalized r-factor infinite log-concave} \textit{iff}
$\mathcal{L}_r^i(a_k)\ $is non negative for all  $i\geq 1$.

D. Uminsky and K. Yeats \cite{p1} studied the properties of the \emph{log-operator} $\mathcal{L}$ on the subset of the finite symmetric sequences of the form
\begin{equation*}
\lbrace \dots ,0,0,1,x_\circ,x_1,\dots,x_n,\dots,x_1,x_\circ,1,0,0,\dots \rbrace,
\end{equation*}
\begin{equation*}
\lbrace \dots ,0,0,1,x_\circ,x_1,\dots,x_n,x_n,\dots,x_1,x_\circ,1,0,0,\dots \rbrace.
\end{equation*}

The first sequence above is referred as odd of length $2n+3$ and second as even of length $2n+4$. Any such sequence corresponds to a point  $\left(x_0,x_1,x_2,\dots,x_{n+1}\right)$ in $\mathbb{R}^{n+1}.$ They prove the  existence of  an infinite region $\mathcal{R} \subset \mathbb{R}^{n+1},$  bounded by $n+1$ parametrically defined hypersurfaces such that any sequence corresponding a point of $\mathcal{R}$ is \emph{infinitely log concave}.

In the first part of this paper, we study the properties of the \emph{Generalized r-factor log-operator} $\mathcal{L}_r$ on these finite symmetric sequences and redefine the parametrically defined hypersurfaces which generalizes the one defined by \cite{p1}. We show that any sequence corresponding a point of the region $\mathcal{R}$, bounded by the new generalized parametrically defined r-factor hypersurfaces, is \emph{Generalized r-factor infinite log concave}.

In the end, we give an improved value of $r_\circ$ found by McNamara and Sagan \cite{p2} as the log-concavity criterion using the new log-operator $\mathcal{L}_r$.
\begin{lem}
Let $(a_k)$ be a r-factor log-concave sequence of non-negative terms. If $\mathcal{L}_r(a_k)$ is Generalized r-factor log-concave $\mathit{then}$
\[(r^5){a_{k-2}\ a_{k-1} \ a_{k+1} \ a_{k+2}}\ \leq\ a_k^4 .\]
In general,\ if
$\mathcal{L}^{i+1}(a_k)$ is Generalized r-factor log-concave, then
\[(r^5)\mathcal{L}_r^i(a_{k-2})\ \mathcal{L}_r^i(a_{k-1})\ \mathcal{L}_r^i(a_{k+1})\ \mathcal{L}_r^i(a_{k+2})\ \leq \ [\mathcal{L}_r^i(a_k)]^4.\]
\end{lem}

\begin{proof}
Let $\mathcal{L}_r(a_k)$ is \emph{r-factor log-concave}, then
\begin{align*}
[\mathcal{L}_r(a_k)]^2 \  &\geq \ r\ [\mathcal{L}_r(a_{k-1})]\ [\mathcal{L}_r(a_{k+1})] \\
(a_k^2-ra_{k-1} \ a_{k+1})^2 \  &\geq \ r\ (a_{k-1}^2-ra_{k-2} \ a_k)\ (a_{k+1}^2-ra_k \ a_{k+2})\\
%
%
\left. \begin{array}{l} a_k^4+(r^2-r)a_{k-1}^2 \ a_{k+1}^2 \\ \ +r^2\ a_{k-1}^2\ a_k\ a_{k+2}+r^2\ a_{k-2}\ a_k \ a_{k+1}^2 \end{array} \right)  &\geq  \ 2\ r\ a_{k-1}\ a_k^2 \ a_{k+1}+r^3\ a_{k-2}\ a_k^2\ a_{k+2}. \\
\intertext{Since $(a_k)$ is \emph{r-factor log concave}, so applying \ $a_k^2\  \geq r\ a_{k-1} \ a_{k+1}$\ , we have}
%
%
(\frac{2r+1}{r})\ a_k^4\  &\geq \ (2r+1)r^4\ a_{k-2}\ a_{k-1}\ a_{k+1}\ a_{k+2} \\
\Rightarrow  \hspace{0.5in}(r^5)\  a_{k-2}\ a_{k-1}\  a_{k+1}\ a_{k+2}\ &\leq \ a_k^4.
\end{align*}
Similarly, if $\mathcal{L}_r^2(a_k)$\ is \emph{Generalized r-factor log-concave, then}
\[(r^5)\ \mathcal{L}_r(a_{k-2})\ \mathcal{L}_r(a_{k-1})\mathcal{L}_r(a_{k+1})\mathcal{L}_r(a_{k+2})\ \leq \ [\mathcal{L}_r(a_k)]^4\]
Continuing this way, if $\mathcal{L}_r^{i+1}(a_k)$\ is \emph{Generalized r-factor log-concave, then}
\[(r^5)\ \mathcal{L}_r^i(a_{k-2})\ \mathcal{L}_r^i(a_{k-1})\ \mathcal{L}_r^i(a_{k+1})\ \mathcal{L}_r^i(a_{k+2})\ \leq \ [\mathcal{L}_r^i(a_k)]^4.\]
\end{proof}
If we can prove conversely, above lemma can be used as an alternative criterion to verify the \emph{r-factor i-fold log-concavity} of a given \emph{r-factor log-concave} sequence.

The Generalized r-factor log-operator $\mathcal{L}_r$ equals the log-operator $\mathcal{L}$ for $r=1$, so \emph{Generalized r-factor infinite log-concavity} implies \emph{infinite log-concavity}.
Thus, we have the following results:
\begin{lem}
Let $(a_k)$ be a log-concave sequence of non-negative terms. If $\mathcal{L}(a_k)$ is log-concave, then
\[{a_{k-2}\ a_{k-1} \ a_{k+1} \ a_{k+2}}\ \leq\ a_k^4 .\]
In general, if
$\mathcal{L}^{i+1}(a_k)$\ is log-concave, then
\[\mathcal{L}^i(a_{k-2})\ \mathcal{L}^i(a_{k-1})\ \mathcal{L}^i(a_{k+1})\ \mathcal{L}^i(a_{k+2})\ \leq \ [\mathcal{L}^i(a_k)]^4\]
\end{lem}
\begin{lem}
Every Generalized r-factor infinitely log-concave sequence $(a_k)$ of non-negative terms is infinitely log-concave.
\end{lem}
\section{Region of Infinite log-concavity and r-factor}

One dimensional even and odd sequences $\left\lbrace 1,x,x,1 \right\rbrace$,$\left\lbrace 1,x,1 \right\rbrace$ correspond to a point $x\in \mathbb{R}$. Uminsky and Yeats \cite{p1} after applying the \emph{log-operator} $\mathcal{L}$  showed that the positive fixed point for the sequence $\mathcal{L}\left\lbrace 1,x,x,1 \right\rbrace = \left\lbrace 1,x^2-x,x^2-x,1 \right\rbrace$ is $x=2$ and for $\mathcal{L}\left\lbrace 1,x,1 \right\rbrace = \left\lbrace 1,x^2-1,1 \right\rbrace$ is $x=\frac{1+\sqrt{5}}{2}$. Also the sequence  $\left\lbrace 1,x,x,1 \right\rbrace$ is infinitely log-concave if $x\geq 2$  and  $\left\lbrace 1,x,1 \right\rbrace$ is infinitely log-concave if $x\geq \frac{1+\sqrt{5}}{2}$. For detail see \cite{p1}

Now if we apply the Generalized r-factor log operator $\mathcal{L}_r$, instead of applying the log operator $\mathcal{L}$, then after a simple calculation we see that the positive fixed point for the sequence
$\mathcal{L}_r\left\lbrace 1,x,x,1 \right\rbrace=\left\lbrace 1,x^2-rx,x^2-rx,1 \right\rbrace$ is $x=1+r$ and for $\mathcal{L}_r\left\lbrace 1,x,1 \right\rbrace=\left\lbrace 1,x^2-r,1 \right\rbrace$ is $x=\frac{1+\sqrt{1+4r}}{2}$. Also the sequence  $\left\lbrace 1,x,x,1 \right\rbrace$ is Generalized r-factor infinitely log-concave if $x\geq {1+r}$ and $\left\lbrace 1,x,1 \right\rbrace$ is Generalized r-factor infinitely log-concave if $x\geq \frac{1+\sqrt{1+4r}}{2}$. This agrees with the results obtained by Uminsky and Yeats for $r=1$.
\subsection{Leading terms analysis using r-factor log-concavity}
Consider the even sequence of length $2n+4$
\begin{footnotesize}
\begin{equation} \label{c21}
s=\bigg \{ 1, a_\circ x, a_1x^{1+d_1}, a_2x^{1+d_1+d_2},\dots, a_nx^{1+d_1+\dots+d_n}, a_nx^{1+d_1+\dots+d_n},\dots, a_1x^{1+d_1},a_\circ x,1 \bigg \}
\end{equation}
\end{footnotesize}
If we apply $\mathcal{L}_r$ operator on s, instead of applying $\mathcal{L}$, then \\
\begin{footnotesize}
\begin{align}
\mathcal{L}_r(&s)=\bigg \{ 1,x(a_\circ^2 x-ra_1x^{d_1}), x^{2+d_1}(a_1^2x^{d_1}-ra_2a_\circ x^{d_2}), x^{2+2d_1+d_2}(a_2^2x^{d_2}-ra_3a_1x^{d_3}),\dots, \notag \\
& x^{2+2d_1+\dots+2d_{n-1}+d_n}(a_n^2x^{d_n}-ra_na_{n-1}), x^{2+2d_1+\dots+2d_{n-1}+d_n}(a_n^2x^{d_n}-ra_na_{n-1}),\dots,1 \bigg \} \notag
\end{align}
\end{footnotesize}
where, $ 0\leq d_n\leq d_{n-1}\leq \dots \leq d_1\leq 1.$
The $(n-1)$\ faces are defined by $d_1=1,\ d_j=d_{j+1},\ $ for\ $0<j<n,$\ and $d_n=0$, they define the boundaries of what will be our open region of convergence, for detail see \ \cite{p1} \\ \\
\textbf{\underline {For $\mathbf{d_1=1}$.}}\quad The leading terms of $\mathcal{L}_r(s)$\ are
\begin{equation*}
\bigg \lbrace 1,(a_\circ^2-ra_1)x^2, a_1^2x^4,a_2^2 x^{4+2d_2},\dots,a_n^2x^{4+2d_2+\dots+2d_n},a_n^2x^{4+2d_2+\dots+2d_n},\dots,1 \bigg\rbrace
\end{equation*}
matching the coefficients of leading terms in \ $\mathcal{L}_r\left(s\right)$\ with the coefficients of $s$. So that the leading terms of \ $\mathcal{L}_r$\ have the same form as $s$ itself for some new $x \ i.e.,$
\begin{center}
$
\begin{array}{l|l}
a_\circ^2-ra_1=a_\circ  & \ \Rightarrow \quad a_\circ=\frac{1+\sqrt{1+4r}}{2} \quad \quad \left(a_1=1, \mathrm{~follows\ from\ the\ next\ step}\right)\\
a_1^2=a_1  & \ \Rightarrow \quad a_1=1 \\
a_2^2=a_2  & \ \Rightarrow  \quad a_2=1 \\
\quad \ \vdots &  \qquad \quad \quad \vdots \\
a_n^2=a_n  & \ \Rightarrow \quad a_n=1

\end{array}
$
\end{center}
so we have the positive values
\begin{equation}
 a_\circ=\frac{1+\sqrt{1+4r}}{2},\qquad \mathrm{and} \qquad a_i=1\quad \mathrm{for} \quad0<i\leq n.         \label{c22}
\end{equation}
This agrees with the values, $a_\circ=\frac{1+\sqrt{5}}{2}, $ and  $a_i=1$ for $0<i\leq n$, obtained by Uminsky and Yeats \cite{p1} for $r=1$.\\\\
\textbf{\underline {For $\mathbf{d_j=d_{j+1}}$.}}\quad The leading terms of $\mathcal{L}_r(s)$\ are
\begin{align*}
& \bigg\lbrace 1,a_\circ^2x^2, a_1^2x^{2+2d_1},a_2^2 x^{2+2d_1+2d_2},\dots,(a_j^2-ra_{j-1}a_{j+1})x^{2+2d_1+\dots+2d_j},  \\ & \quad a_{j+1}^2x^{2+2d_1+\dots+2d_{j-1}+4d_j},\dots, a_n^2 x^{2+2d_1+\dots+2d_n},a_n^2 x^{2+2d_1+\dots+2d_n},\dots,1 \bigg\rbrace
\end{align*}
comparing the coefficients, we get the positive values
\begin{equation}
a_i=1 \quad \mathrm{for} \quad i\neq j,\qquad \mathrm{and} \qquad a_j=\frac{1+\sqrt{1+4r}}{2}.         \label{c23}
\end{equation}
This gives the values for $r=1$, $a_i=1 \  \mathrm{for} \ i\neq j, $ and $a_j=\frac{1+\sqrt{5}}{2}$, same as obtained in \cite{p1}. \\ \\
\textbf{\underline {For $\mathbf{d_n=0}$}.}\quad The leading terms of $\mathcal{L}_r(s)$\ are
\begin{align*}
& \bigg\lbrace 1,a_\circ^2x^2, a_1^2x^{2+2d_1},a_2^2 x^{2+2d_1+2d_2},\dots,a_{n-1}^2 x^{2+2d_1+\dots+2d_{n-1}}, \notag \\ & (a_n^2-r a_n a_{n-1}) x^{2+2d_1+\dots+2d_{n-1}},(a_n^2-r a_n a_{n-1}) x^{2+2d_1+\dots+2d_{n-1}},\dots,1 \bigg\rbrace
\end{align*}
compating the coefficients, we get the values
\begin{equation}
a_i=1 \quad \mathrm{for} \quad 0\leq i<n, \qquad \mathrm{and} \qquad a_n=1+r.       \label{c24}
\end{equation}
This again agrees with the values, $a_i=1 \quad \mathrm{for} \quad 0\leq i<n,\quad $and \quad $a_n=2$, obtained in \cite{p1} for $r=1$. \\
Similarly for the odd sequence of length $2n+3$
\begin{equation}  \label{c25}
s=\bigg\lbrace 1, a_\circ x, a_1 x^{1+d_1}, a_2 x^{1+d_1+d_2},\dots, a_nx^{1+d_1+\dots+d_n},\dots, a_1x^{1+d_1},a_\circ x,1 \bigg\rbrace
\end{equation}
applying $\mathcal{L}_r$ operator
\begin{small}
\begin{align*}
\mathcal{L}_r(s) &=\bigg\lbrace 1,x(a_\circ^2 x-ra_1x^{d_1}), x^{2+d_1}(a_1^2x^{d_1}-ra_2a_\circ x^{d_2}), x^{2+2d_1+d_2}(a_2^2x^{d_2}-ra_3a_1x^{d_3})  \notag \\
& \qquad ,\dots,x^{2+2d_1+\dots+2d_{n-1}}(a_n^2x^{2d_n}-ra_{n-1}^2),\dots,1 \bigg\rbrace
\end{align*}
\end{small}
\textbf{For $\mathbf{d_1=1\ \mathrm{and}\ d_j=d_{j+1}}$}
This is equivalent to the even case. See (\ref{c22}),\ (\ref{c23}).\\ \\
So we only analyze for $\mathbf{d_n=0}$.
The leading terms of $\mathcal{L}_r(s)$\ are
\begin{align*}
\left\lbrace 1,a_\circ^2x^2, a_1^2x^{2+2d_1},\dots,a_{n-1}^2 x^{2+2d_1+\dots+2d_{n-1}},(a_n^2-ra_{n-1}^2) x^{2+2d_1+\dots+2d_{n-1}},\dots,1 \right\rbrace
\end{align*}
so equating the coefficients, we get, 
\begin{equation} \label{c2odd1}
a_i=1 \quad \mathrm{for} \quad 0\leq i<n, \qquad \mathrm{and} \qquad a_n=\frac{1+\sqrt{1+4r}}{2}.
\end{equation}
This again agrees with the values for $r=1$, as obtained in \cite{p1}.
\subsection{r-factor Hypersurfaces}
The even sequence (\ref{c21}) and the odd sequence (\ref{c25})  correspond to the point \\
$\left(a_\circ x, a_1x^{1+d_1},\dots, a_nx^{1+d_1+\dots+d_n} \right) \in \mathbb{R}^{n+1}$. Hence from $\left(\ref{c22}\right),\left(\ref{c23}\right), \left(\ref{c24}\right)$ and (\ref{c2odd1}) the redefined and generalized parametrically defined Hypersurfaces are

\begin{footnotesize}
\begin{align*}
\mathcal{H}_\circ &= \Bigg\lbrace \left(\frac{1+\sqrt{1+4r}}{2}x, x^2,x^{2+d_2},\dots,x^{2+d_2+\dots+d_n}\right) : 1\leq x,\ 1>d_2>\dots>d_n>0  \Bigg\rbrace \\ \\
\mathcal{H}_j &= \Bigg\lbrace \left( x,x^{1+d_1},\dots,\frac{1+\sqrt{1+4r}}{2}\  x^{1+d_1+\dots+d_j}, x^{1+d_1+\dots+d_{j-1}+2d_j},\dots, \right. \\ & \qquad \quad x^{1+d_1+\dots+d_{j-1}+2d_j+d^{j+2}+\dots+d_n}  \bigg): 1\leq x,\ 1>d_1>\dots>d_j>d_{j+2}>\dots>d_n>0 \Bigg\rbrace
\end{align*}
\end{footnotesize}
The hypersurfaces $\mathcal{H}_j$ are same  for $0 \leq j<n$  in both even and odd cases, while $\mathcal{H}_n$ is different i.e.,
\begin{footnotesize}
\begin{align*}
\intertext{In even case:}
\mathcal{H}_n &= \Bigg\lbrace \bigg( x,x^{1+d_1},\dots, x^{1+d_1+\dots+d_{n-1}},(1+r)x^{1+d_1+\dots+d_{n-1}} \bigg) : 1\leq x,\ 1>d_1>\dots>d_{n-1}>0  \Bigg\rbrace \\
\intertext{In odd case:}
\mathcal{H}_n &= \Bigg\lbrace \left( x,x^{1+d_1},\dots, x^{1+d_1+\dots+d_{n-1}},\frac{1+\sqrt{1+4r}}{2} x^{1+d_1+\dots+d_{n-1}} \right) : 1\leq x,\ 1>d_1>\dots>d_{n-1}>0 \Bigg\rbrace
\end{align*}
\end{footnotesize}
Hence the r-factor hypersurfaces define the general case which  for $r=1$ agrees with the hypersurfaces obtained in \cite{p1}. 
 
 So from here onward we consider $\mathcal{R}$ to be the region of Generalized r-factor infinite log-concavity and is bounded by the new generalized r-factor hypersurfaces. 
 
 Also any sequence
$
\lbrace \dots ,0,0,1,x_\circ,x_1,\dots,x_n,x_n,\dots,x_1,x_\circ,1,0,0,\dots \rbrace
$
is in $\mathcal{R}$ \textit{iff} $(x_\circ,x_1,\dots,x_n)\in \mathcal{R}$ and with the positive increasing coordinates defined as greater in the $i^{\mbox{th}}$ coordinate than $\mathcal{H}_i$. In this case we say that above sequence lies on the correct side of $\mathcal{H}_i$. For detail see  \cite{p1}. \\
Next we present the r-factor log-concavity version of the Lemma $\left(3.2\right)$ of \cite{p1}.
\begin{lem}                      \label{c2lem}
Let the sequence
\[s= \bigg\lbrace 1,x,x^{1+d_1},x^{1+d_1+d_2},\dots,x^{1+d_1+\dots+d_n},x^{1+d_1+\dots+d_n},\dots,x,1 \bigg\rbrace\]
be r-factor 1-log-concave for $x>0$. Then $1\geq d_1 \geq \dots \geq d_n \geq 0.$
\end{lem}
\begin{proof}
Using definition of r-factor log-concavity, it can be easily proved.
\end{proof}
A similar results holds for the odd sequence as well.

In Lemma $\left(3.3\right)$, Uminsky and Yeats \cite{p1} using properties of the triangular numbers and the sequence
\begin{equation}
s=\bigg\lbrace 1,C^{T(0)}ax_\circ,C^{T(1)}a^2x_1,C^{T(2)}a^3x_2,\dots,C^{T(n)}a^{n+1}x_n,C^{T(n)}a^{n+1}x_n,\dots,1 \bigg\rbrace  \label{c2seqTs}
\end{equation}
proved the existence of the log-concavity region $\mathcal{R}$ by applying log-operator $\mathcal{L}$ for $a>2C^{T(n-1)-T(n)}$ and for $0<C<\frac{2}{1+\sqrt{5}}$. 

If we choose $C$ such that
$0<C<\frac{2\sqrt{r}}{1+\sqrt{1+4r}}$, then applying the Generalized r-factor log-operator $\mathcal{L}_r$ on the sequence (\ref{c2seqTs}), we can easily prove the existence of the Generalized r-factor log-concavity region $\mathcal{R}$ for  $a>(1+r)C^{T(n-1)-T(n)}$.

Sequence $s$\ (\ref{c2seqTs}) is not the only sequence for which $\mathcal{R}$ is non-empty. For completeness, we will give an alternative proof of the  Lemma~(\ref{c2lemNonEmpty}) using Pentagon numbers. One can also prove it by some other numbers such as figurate numbers. \\

%
%
%
%
%
Let $\tilde{P}(n)$ denotes the $n^{th}$ pentagonal number, then
\begin{align*}
\hspace{1in} \tilde{P}(n)&=\frac{n(3n-1)}{2} = \tilde{P}(n-1)+3n-2 \hspace{1in}
\end{align*}
Define $P(n)=2\tilde{P}(n)$ for $n\geq 0$, we can easily have
\begin{align}
P(n+1)+P(n-1)&= 2P(n)+6         \label{c2p1}    \\
\Rightarrow \hspace{0.8in} P(n+1)+P(n-1)&> 2P(n)    \label{c2p2}    \\
\Rightarrow \hspace{1.23in} C^{P(n+1)+P(n-1)}&< C^{2P(n)} \hspace{0.5in} \mathrm{for\ all}\ \ C<1   \label{c2p3}
\end{align}
\begin{equation}
\mathrm{Also} \qquad \qquad P(0)-\frac{P(1)}{2}=-1   \qquad \because \ \tilde{P}(0)=0\ \mathrm{and} \ \tilde{P}(1)=1            \label{c2p4}
\end{equation}

Hence the Generalized r-factor log-concavity version of Lemma $\left(3.3\right)$ of \cite{p1} is given below:
\begin{lem}                 \label{c2lemNonEmpty}
The Generalized r-factor infinite log-concavity region $\mathcal{R}$ is non-empty and unbounded.
\end{lem}

\begin{proof} Let $q=\lbrace \dots ,0,0,1,x_\circ,x_1,\dots,x_n,\dots,x_1,x_\circ,1,0,0,\dots \rbrace$\ \label{e0} be any r-factor log-concave sequence of positive terms.
\\ 
Choose $C$ such that
\begin{equation}              \label{c2lemP1}
0<C<\frac{2\sqrt{r}}{1+\sqrt{1+4r}}<1
\end{equation}
and consider the following sequence
\begin{align}
s&=\bigg\lbrace 1,C^{P(0)}ax_\circ,C^{P(1)}a^2x_1,C^{P(2)}a^3x_2,\dots,C^{P(n)}a^{n+1}x_n,C^{P(n)}a^{n+1}x_n,\dots,1 \bigg\rbrace  \notag
\\ &
\mathrm{for}\  a>(1+r)C^{P(n-1)-P(n)}>C^{P(n-1)-P(n),} \label{c2lemP2}
\end{align}
now using r-factor log-concavity of $q$, we have
\begin{align}
C^{2P(0)}a^2x_\circ^2\ &=\ a^2x_\circ^2\geq a^2rx_1>rC^{P(1)}a^2x_1 \hspace{0.75in}  \label{c2lemP3}
\intertext{also for $0<j<n$}
C^{2P(j)} a^{2j+2}x_j^2\  &\geq \ C^{2P(j)} a^{2j+2}(r x_{j-1} x_{j+1})\notag \\
&=\ r\ C^{2P(j)} \ a^j x_{j-1}\ \  a^{j+2}x_{j+1} \notag \\
&>\ r\ C^{P(j-1)} a^j x_{j-1}\ \  C^{P(j+1)}a^{j+2}x_{j+1}.  \qquad \mathrm{by}~(\ref{c2p3})            \label{c2lemP4}
\intertext{Now consider}
C^{P(n)} a^{n+1}x_n\  &\geq \ a\ C^{P(n)} a^n(r x_{n-1})   \notag  \\
&>\ C^{P(n-1)-P(n)}\ r\ C^{P(n)} a^n x_{n-1} \qquad \qquad \ \mathrm{by}~ (\ref{c2lemP2}) \notag \\
&>\ r\ C^{P(n-1)} a^n x_{n-1} \notag \\
&>\ C^{P(n-1)} a^n x_{n-1}  \label{c2lemP5}
\intertext{and so}
C^{2P(n)} a^{2n+2}x_n^2\ &=\ C^{P(n)} a^{n+1} x_n\ C^{P(n)}\ a^{n+1} x_n \notag \\
&>\ r\ C^{P(n-1)} a^n x_{n-1} \ C^{P(n)} a^{n+1} x_n.  \qquad \mathrm{by}~ (\ref{c2lemP5})  \label{c2lemP6}
\end{align}
From (\ref{c2lemP3}),(\ref{c2lemP4}),(\ref{c2lemP6}), we conclude that $s$ is also r-factor 1-log-concave. 

Define $\tilde{x}=C^{P(0)}ax_\circ$ and define $\tilde{d_1}$ such that $\tilde{x}^{1+\tilde{d_1}}=C^{P(1)}a^2x_1$ and continuing, we have $\tilde{x}^{1+\tilde{d_1}+\dots+\tilde{d_j}}=C^{P(j)}a^{j+1}x_j$ \\ \\
$\Rightarrow$ $\hspace{1.5in} 1>\tilde{d_1}>\tilde{d_2}>\dots>\tilde{d_n}>0$  \qquad \quad by lemma~ (\ref{c2lem})  \\ \\
%
%
%
\underline{\textbf{For} ${\mathbb{\mathcal{H}}_j}$}\\
Choose $x=\tilde{x},\ d_i=\tilde{d_i}$ for $i\neq j,{j+1}.$  and $d_j=(\tilde{d_j}+\tilde{d_{j+1}})/2 $ \ for hypersurface $\mathcal{H}_j.$ \\
Consequently \hspace{0.2in} $1>d_1>\dots>d_j>d_{j+2}>\dots>d_n>0, $ and so
\begin{align}
C^{P(j)} a^{j+1}x_j\  &\geq \ C^{P(j)} a^{j+1} \sqrt{r x_{j-1} x_{j+1}} \notag \\
&=\ \sqrt{r}\ \sqrt{C^{2P(j)-P(j+1)-P(j-1)} C^{P(j-1)} a^j x_{j-1}\ \  C^{P(j+1)}a^{j+2}x_{j+1}} \notag \\
&=\ \sqrt{r}\ \sqrt{C^{-6}\ x^{1+d_1+\dots+d_{j-1}}\ x^{1+d_1+\dots+d_{j-1}+2d_j}}   \qquad \qquad \mathrm{by}~(\ref{c2p1}) \notag \\
&=\ \sqrt{r}\ C^{-3}\ x^{1+d_1+\dots+d_{j-1}+d_j}  \notag  \\
&>\ \sqrt{r}\ C^{-1}\ x^{1+d_1+\dots+d_{j-1}+d_j} \notag  \\
&>\ \frac{1+\sqrt{1+4r}}{2}\ \ x^{1+d_1+\dots+d_{j-1}+d_j}  \hspace{0.7in} \mathrm{by}~ (\ref{c2lemP1})     \label{c2lemP7}
\end{align}
Thus $s$ is on the correct side of $\mathcal{H}_j$.
%
%
%
\\ \\
\underline{\textbf{For} ${\mathbb{\mathcal{H}_\circ}}$}\\
Choose $x=\tilde{x},\  d_1=1$ and $\ d_i=\tilde{d_i}$ $\forall$ $i>1$. \\ \\
Consequently, $1>d_2>\dots>d_n>0 $, \quad \qquad \quad by lemma~ (\ref{c2lem})  \\
and so
\begin{align}
C^{P(1)} a^2 x^1 &=\tilde{x}^{1+\tilde{d_1}}=\tilde{x}^2=x^2 \hspace{2in} \notag     \\
\Rightarrow \hspace{1in} a^2x_1 &=C^{-P(1)} x^2  \label{c2lemP8}      \\
\mathrm{also} \hspace{1in} C^{P(j)} a^{j+1}x^j &=\tilde{x}^{1+\tilde{d_1} +\dots+\tilde{d}_j} = x^{2+d_2+\dots+d_j} \notag
\end{align}
Now we check
\begin{align}
C^{P(0)} a x_\circ \  &\geq \ C^{P(0)} \sqrt{r a^2 x_1}  \notag \\
&=\ \sqrt{r}\ C^{P(0)}\sqrt{C^{-P(1)} x^2}\hspace{1.1in} \mathrm{by}~ (\ref{c2lemP8})     \notag \\
&=\ \sqrt{r}\ C^{P(0)-\frac{-P(1)}{2}}\ x  \notag\\
&=\ \sqrt{r}\ C^{-1}\ x  \hspace{1.8in}  \mathrm{by}~ (\ref{c2p4})  \notag \\
&>\ \frac{1+\sqrt{1+4r}}{2}\ \ x  \hspace{1.42in}  \mathrm{by}~ (\ref{c2lemP1})      \label{c2lemP9}
\end{align}
Thus $s$ is on the correct side of $\mathcal{H}_0$.
%
%
%
\\ \\
\underline{\textbf{For} ${\mathbb{\mathcal{H}}_n}$}\\
Choose $x=\tilde{x},\ $ and $\ d_i=\tilde{d_i}$ for $i<n,\ \tilde{d}_n = d_n = 0$  for $\mathcal{H}_n$.\\
Consequently, we have \hspace{0.5in} $1>d_1>\dots>d_{n-1}>0 $, 
\begin{align}
C^{P(n)} a^{n+1} x_n \  &\geq \ C^{P(n)} a^{n+1} (r\ x_{n-1})  \notag \\
&=\ r\ a\ C^{P(n)-P(n-1)} \ C^{P(n-1)} \ a^n \ x_{n-1}  \notag \\
&\geq \ a\ C^{P(n)-P(n-1)} \ x^{1+d_1+\dots+d_{n-1}} \notag\\
&> \ (1+r)\ x^{1+d_1+\dots+d_{n-1}} \hspace{0.9in} \mathrm{by\ }(\ref{c2lemP2}) \label{c2lemP10}
\end{align}
Thus $s$ is on the correct side of $\mathcal{H}_n$.

From (\ref{c2lemP7}),(\ref{c2lemP9}),(\ref{c2lemP10}), and by the definition of the region $\mathcal{R}$, we conclude that sequence $s$ is in $\mathcal{R}$.
Hence using \emph{r-factor log-concavity}, $\mathcal{R}$ is non-empty and unbounded.\\
\end{proof}
\noindent For the same choice of $C$ (\ref{c2lemP1}), lemma (\ref{c2lemNonEmpty})~also holds for the odd sequence
\begin{align*}
s &= \bigg\lbrace 1,C^{P(0)} ax_\circ, C^{P(1)}a^2x_1, C^{P(2)}a^3 x_2,\dots,C^{P(n)} a^{n+1} x_n, C^{P(n-1)}a^n x_n,\dots,1 \bigg\rbrace  \notag
\\ &
\mathrm{for}\  a> \left(\frac{1+\sqrt{1+4r}}{2\sqrt{r}}\right)C^{P(n-1)-P(n)}>C^{P(n-1)-P(n)}.
\end{align*}
So $\mathcal{R}$ is non-empty and unbounded in the odd case as well.
\subsection{Region of Generalized r-factor Infinite Log-Concavity}
Now we present the Generalized r-factor Infinite log-concavity version of the main theorem of \cite{p1}.
%
%
%
%
\begin{thm} \label{maintheorem}
Any sequence in $\mathcal{R}$ is Generalized r-factor Infinite log-concave.
\end{thm}
%
%
\begin{proof}
By definition of $\mathcal{R}$\\
Let
\begin{small}
\begin{align*}
s &=\bigg\lbrace 1, x, x^{1+d_1},\dots,x^{1+d_1+\dots+d_{j-1}},\frac{1+\sqrt{1+4r}}{2}  x^{1+d_1+\dots+d_j}+\epsilon, x^{1+d_1+\dots+d_{j-1}+2d_j}, \\
& \quad \quad x^{1+d_1+\dots+2d_j+\dots+d_n}, x^{1+d_1+\dots+2d_j+\dots+d_n},\dots,1 \bigg\rbrace \qquad x, \epsilon >0
\end{align*}
\end{small}
be a sequence in $\mathcal{R}$ \\ \\
Applying $\mathcal{L}_r$ operator on $s$ and simplifying, we get
\begin{footnotesize}
\begin{align*}
& \mathcal{L}_r(s)= \\
& \Bigg\lbrace 1,\ x^2-rx^{1+d_1},\ \dots,\ x^{2+2d_1+\dots+2d_{j-1}}-r \left(\frac{1+\sqrt{1+4r}}{2} \right)  x^{2+2d_1+\dots+2d_{j-2}+d_{j-1}+d_j}   \\
& -\epsilon \ r\ x^{1+d_1+\dots+d_{j-2}},\
\left(\left( \frac{1+\sqrt{1+4r}}{2} \right)^2-r \right)  x^{2+2d_1+\dots+2d_j}+\epsilon^2 \\
& - \epsilon \left(1+\sqrt{1+4r}\right) x^{1+d_1+\dots+d_j}, \ x^{2+2d_1+\dots+2d_{j-1}+4d_j}-r\left(\frac{1+\sqrt{1+4r}}{2}\right)  x^{2+2d_1+\dots+3d_j+d_{j+2}} \\
& -r\ \epsilon \left(x^{1+d_1+\dots+2d_j+d_{j+2}} \right),\dots,  x^{2+2d_1+\dots+4d_j+\dots+2d_n}-r\left(x^{2+2d_1+\dots+4d_j+\dots+2d_{n-1}+d_n} \right), \\
&  x^{2+2d_1+\dots+4d_j+\dots+2d_n}-r\left(x^{2+2d_1+\dots+4d_j+\dots+2d_{n-1}+d_n} \right),\dots,1 \Bigg\rbrace
\end{align*}
\end{footnotesize}
\begin{equation} \label{r.equation}
\mathrm{Since}\hspace{0.9in}\left(\frac{1+\sqrt{1+4r}}{2}\right)^2-r\ =\ \frac{1+\sqrt{1+4r}}{2}, \hspace{1in}
\end{equation}
so by using $x^2$ in place of $x$ in the definition of $\mathcal{H}_j$ and applying Lemma(3.4)~of \cite{p1}, we conclude that both $s$ and $\mathcal{L}_r(s)$ are on the same side of $\mathcal{H}_j$ which are larger in the $j^{\mbox{th}}$ coordinate.
Hence result is true for hypersurface $\mathcal{H}_j$. \\ 
%
%
Similarly, for $x, \epsilon >0$ consider the sequence
\begin{align*}
s =\left \{ 1, \frac{1+\sqrt{1+4r}}{2} x+\epsilon, x^2,\dots,x^{2+d_2+\dots+d_n},  x^{2+d_2+\dots+d_n},\dots,1 \right \}
\end{align*}
After applying $\mathcal{L}_r$ operator on $s$ and simplifying, we get
\begin{align*}
& \mathcal{L}_r(s)= \\
& \Bigg\lbrace 1,\left(\left( \frac{1+\sqrt{1+4r}}{2} \right)^2-r \right)  x^2
+ \epsilon \left(1+\sqrt{1+4r}\right) x +\epsilon^2 ,  \\
& x^4 - r \left(\frac{1+\sqrt{1+4r}}{2}\right)x^{3+d_2}-r \epsilon x^{2+d_2},\dots,x^{4+2d_2+\dots+2d_n} - r x^{4+2d_2+\dots+2d_{n-1}+d_n},\\
& x^{4+2d_2+\dots+2d_n} - r x^{4+2d_2+\dots+2d_{n-1}+d_n} ,\dots,1 \Bigg\rbrace
\end{align*}
again by (\ref{r.equation}) and Lemma(3.4)~of \cite{p1}, we conclude that $s$ and $\mathcal{L}_r(s)$ lie on the same side of $\mathcal{H}_\circ$.
Hence result is true for $\mathcal{H}_\circ$.\\ \\
%
%
%
%
Finally, for $x, \epsilon >0,\quad d_n=0$ consider the sequence in $\mathcal{R}$
\begin{footnotesize}
\begin{align*}
s =\bigg\lbrace 1,x,x^{1+d_1},\dots,x^{1+d_1+\dots+d_{n-1}}, \left(1+r\right) x^{1+d_1+\dots+d_{n-1}}+\epsilon,\left(1+r\right) x^{1+d_1+\dots+d_{n-1}}+\epsilon,\dots,1 \bigg\rbrace
\end{align*}
\end{footnotesize}
Applying $\mathcal{L}_r$, we get
\begin{small}
\begin{align*}
& \mathcal{L}_r(s)= \\
& \bigg\lbrace 1,x^2-rx^{1+d_1}, x^{2+2d_1}-r x^{2+d_1+d_2},\dots,
x^{2+2d_1+\dots+2d_{n-1}} -r(1+r) x^{2+2d_1+\dots+2d_{n-2}+d_{n-1}} \\
& -\epsilon r x^{1+d_1+\dots+d_{n-2}}, \left((1+r)^2-r(1+r)\right) x^{2+2d_1+\dots+2d_{n-1}}+\epsilon (r+2) x^{1+d_1+\dots+d_{n-1}}+ \epsilon^2, \\
& \left((1+r)^2-r(1+r)\right) x^{2+2d_1+\dots+2d_{n-1}}+\epsilon (r+2) x^{1+d_1+\dots+d_{n-1}}+ \epsilon^2 ,\dots, 1 \bigg\rbrace
\end{align*}
\end{small}
\begin{equation*}
\mathrm{Since} \hspace{1in}(1+r)^2-r(1+r)\ =\ 1+r, \hspace{2in}
\end{equation*}
so again by Lemma(3.4)~of \cite{p1}, we conclude that $s$ and $\mathcal{L}_r(s)$ lie on the same side of $\mathcal{H}_n$.
Hence the result is true for considering $\mathcal{H}_n$.

Consequently from the above three cases,  \ $s\in \mathcal{R}\  \Rightarrow \ \mathcal{L}_r(s) \in \mathcal{R} $. 
Hence any sequence in $\mathcal{R}$ is Generalized r-factor Infinite log-concave.

In case of the odd sequences, system is equivalent to the even case for $\mathcal{H}_\circ$ and $\mathcal{H}_j$.
So we only need to consider for $\mathcal{H}_n$.
Let 
\begin{small}
\begin{align*}
s =\bigg\lbrace 1,x,x^{1+d_1},\dots,x^{1+d_1+\dots+d_{n-1}}, \frac{1+\sqrt{1+4r}}{2} x^{1+d_1+\dots+d_{n-1}}+\epsilon, x^{1+d_1+\dots+d_{n-1}},\dots,1 \bigg\rbrace
\end{align*}
\end{small}
be a sequence in $\mathcal{R}$. \\
Applying $\mathcal{L}_r$ operator on $s$ and simplifying, we get
\begin{align*}
& \mathcal{L}_r(s)= \\
& \Bigg\lbrace 1,x^2-rx^{1+d_1}, x^{2+2d_1}-r x^{2+d_1+d_2},\dots, x^{2+2d_1+\dots+2d_{n-1}}   \\
& -r \left(\frac{1+\sqrt{1+4r}}{2}\right)x^{2+2d_1+\dots+2d_{n-2}+d_{n-1}} -\epsilon r x^{1+d_1+\dots+d_{n-2}},\\
& \left( \left(\frac{1+\sqrt{1+4r}}{2}\right)^2-r\right) x^{2+2d_1+\dots+2d_{n-1}}+\epsilon \left(1+\sqrt{1+4r}\right) x^{1+d_1+\dots+d_{n-1}}+ \epsilon^2, \\
&  x^{2+2d_1+\dots+2d_{n-1}} -r\left(\frac{1+\sqrt{1+4r}}{2}\right) x^{2+2d_1+\dots+2d_{n-2}+d_{n-1}}-\epsilon r x^{1+d_1+\dots+d_{n-2}},\dots, 1 \Bigg\rbrace
\end{align*}
So by (\ref{r.equation}) and Lemma(3.4)~of \cite{p1}, we conclude that $s$ and $\mathcal{L}_r(s)$ lie on the same side of $\mathcal{H}_\circ$.
Hence any (odd) sequence in $\mathcal{R}$ is also Generalized r-factor Infinite log-concave.
\end{proof}
\section{Generalized r-factor Infinite Log-Concavity Criterion}
We start this section by a Lemma 2.1, proved by McNamara and Sagan \cite{p2}
using the log-operator $\mathcal{L}$, that is
\begin{lem}[Lemma 2.1 of \cite{p2}]    \label{c3lem1}
Let $(a_k)$ be a non-negative sequence and let $r_\circ = (3 +\sqrt{5})$. Then $(a_k)$ being $r_\circ$-factor log-concave implies that $\mathcal{L}(a_k)$ is too. So in this case $(a_k)$ is
infinitely log-concave.
\end{lem}
\begin{proof}
See McNamara and Sagan \cite{p2}.
\end{proof}
If we apply the Generalized r-factor log-operator $\mathcal{L}_r$, instead of applying the log-operator $\mathcal{L}$, we have the following result:
\begin{lem}      \label{improvedvalueof.r}
Let $(a_k)$ be a sequence of non-negative terms and $r=1+\sqrt{2}$. \\
If $(a_k)$ is Generalized r-factor log-concave, then so is $\mathcal{L}_{r}(a_k)$\\
Hence continuing, $(a_k)$ is Generalized $r$-factor infinitely log-concave sequence.
\end{lem}

\begin{proof}
Let $(a_k)$ be \emph{r-factor log-concave} sequence of non-negative terms.
\\
Now $\mathcal{L}_r(a_k)$ will be \emph{r-factor log-concave}, $\mathit{iff}$
\begin{small}
\begin{align}
[\mathcal{L}_r(a_k)]^2 \  &\geq \ r\ [\mathcal{L}_r(a_{k-1})]\ [\mathcal{L}_r(a_{k+1})] \notag \\
(a_k^2-ra_{k-1} \ a_{k+1})^2 \  &\geq \ r\ (a_{k-1}^2-ra_{k-2} \ a_k)\ (a_{k+1}^2-ra_k \ a_{k+2}) \notag\\
%
\left. \begin{array}{l} a_k^4+(r^2-r)a_{k-1}^2 \ a_{k+1}^2 \\ \ +r^2\ a_{k-1}^2\ a_k\ a_{k+2}+r^2\ a_{k-2}\ a_k \ a_{k+1}^2 \end{array} \right)  &\geq  \ 2\ r\ a_{k-1}\ a_k^2 \ a_{k+1}+r^3\ a_{k-2}\ a_k^2\ a_{k+2} \notag \\
\mathrm{or} \notag \\
2\ a_{k-1}\ a_k^2 \ a_{k+1}+r^2\ a_{k-2}\ a_k^2\ a_{k+2}\ &\leq \ \frac{1}{r}a_k^4+(r-1)a_{k-1}^2 \ a_{k+1}^2\notag  \\ &\qquad+r\ a_{k-1}^2\ a_k\ a_{k+2}+r\ a_{k-2}\ a_k \ a_{k+1}^2      \notag  \\
&\leq \ a_k^4+(r-1)a_{k-1}^2 \ a_{k+1}^2  \notag  \\ &\qquad+r\ a_{k-1}^2\ a_k\ a_{k+2}+r\ a_{k-2}\ a_k \ a_{k+1}^2    \label{c31} \
\end{align}
\end{small}
Since $(a_k)$ is \emph{r-factor log concave}, so applying\quad $a_k^2\  \geq r\ a_{k-1} \ a_{k+1}$,\ to the L.H.S. of the above inequality, we have
\begin{align*}
2\ a_{k-1}\ a_k^2 \ a_{k+1}+r^2\ a_{k-2}\ a_k^2\ a_{k+2}\ &\leq \ \frac{2}{r}a_k^4+\frac{1}{r^2}a_k^4= \ \left(\frac{2r+1}{r^2}\right)~a_k^4
\end{align*}
So to keep (\ref{c31})~valid, we have
\begin{align*}
\frac{2r+1}{r^2}&=1 \hspace{2in}  \\
\Rightarrow  \hspace{2in} r^2-2r-1&=0
\end{align*}
$r=1+\sqrt{2}$, is the positive root of the above equation.
This proves the assertion. Thus, if $(a_k)$ is Generalized r-factor log-concave, then so is $\mathcal{L}_r(a_k)$.
Continuing this way, if $\mathcal{L}_r^i(a_k)$ is Generalized r-factor log-concave, then so is $\mathcal{L}_r^{i+1}(a_k)$. This also implies Generalized  r-factor infinite log-concavity of the sequence $(a_k)$.
\end{proof}
Comparing this new value of r, say $r_1=1+\sqrt{2}$, with the value of $r_\circ=\frac{3+\sqrt{5}}{2}$ obtained by McNamara and Sagan \cite{p2}. We find that the value of $r_1=1+\sqrt{2}$ obtained by using Generalized r-factor log-concavity is smaller than obtained by McNamara and Sagan which is $r_\circ=\frac{3+\sqrt{5}}{2}$.

So in this way we get an improved /smaller value of $r=1+\sqrt{2}$. It is clear that Generalized r-factor log concave operator is more useful and dynamic than the previously used log-operator $\mathcal{L}$.
Hence for the new improved value of r, we can restate Lemma (3.1)  \cite{p2} as:

\begin{lem}    \label{revisedlemma3.1}
Let $a_\circ , a_1, \cdots, a_{2m+1}$ be symmetric, nonnegative sequence such that
\begin{enumerate}
\item[(i)] $a^2_k \geq r_1 a_{k-1} a_{k+1}$ \hspace{0.75in} for $k<m$, and
\item[(ii)] $a_m \geq (1+r)a_{m-1} \hspace{0.7in} for\  r \geq 1$
\end{enumerate}
Then $\mathcal{L}_{r_1}(a_k)$ has the same properties, which implies that $(a_k)$ is $r_1$-factor infinitely log-concave.
\end{lem}
\begin{proof}
Using definition of $\mathcal{L}_{r}$, it can be easily proved. For detail see \cite{p2}
\end{proof}
 Using above lemma we now show that Generalized r-factor log-operator $\mathcal{L}_r$ and r-factor hypersurfaces agrees with Theorem (3.2) of \cite{p2} for $r=1$. It also proves theorem (\ref{maintheorem})~alternatively.
\begin{thm}[Revised Theorem 3.2  \cite{p2}]
Any sequence corresponding to a point of $\mathbf{\mathcal{R}}$ is Generalized infinitely $r_1$-factor log-concave.
\end{thm}

\begin{proof}
Let $(a_k)$ be a sequence corresponding to a point of $\mathbf{\mathcal{R}}$. Then, for $(a_k)$,
being on the correct side of $\mathbf{\mathcal{H}_j}$, we have

\begin{align}
a_j\ &\geq \ \left(\frac{1+\sqrt{1+4r}}{2}\right)\ x^{1+d_1+\dots+d_j} \notag
\end{align}
\begin{align}
\Rightarrow \hspace{0.8in}
a_j^2\ &\geq \ \left(\frac{1+\sqrt{1+4r}}{2}\right)^2\ x^{2+2d_1+\dots+2d_j} \notag \\
&= \ \left(\frac{1+2r+\sqrt{1+4r}}{2}\right)\ x^{1+d_1+\dots+d_{j-1}}\  x^{1+d_1+\dots+d_{j-1}+2d_j} \notag \\
&= \ \left(\frac{1+2r+\sqrt{1+4r}}{2}\right)\  a_{j-1}\ a_{j+1} \qquad  \ \mathrm{for}\ 0<j<n \notag
\end{align}
but $r \geq 1$, so above inequality is true for $r=1$ as well
\begin{align}
\Rightarrow \hspace{0.8in}
a_j^2 \  &\geq \ \left(\frac{3+\sqrt{5}}{2}\right)\ a_{j-1}\ a_{j+1}\ =\ r_\circ \ a_{k-1}\ a_{k+1}   \label{rr1} \hspace{0.7in} \\
\Rightarrow \hspace{0.8in} a_j^2 \  &\geq \ \left(1+\sqrt{2}\right)\ a_{j-1}\ a_{j+1}\ =\ r_1 \ a_{j-1}\ a_{j+1}  \label{rr2}
\end{align}
%
%
%
Also being on the correct side of $\mathbf{\mathcal{H}_\circ}$, we have
\begin{align}
a_\circ\ &\geq \ \left(\frac{1+\sqrt{1+4r}}{2}\right)\ x  \notag \\
\Rightarrow \hspace{0.8in}
a_\circ^2\ &\geq \ \left(\frac{1+\sqrt{1+4r}}{2}\right)^2\ x^2 \hspace{2in} \notag \\
&= \ \left(\frac{1+2r+\sqrt{1+4r}}{2}\right)\  a_1 \hspace{2in} \notag \\
\mathrm{also\ true\ for\ r=1} \notag \\
\Rightarrow \hspace{0.8in}
a_\circ^2 \  &\geq \ \left(\frac{3+\sqrt{5}}{2}\right)\ a_1 =\ r_\circ \ a_{-1}\ a_{1} \label{rr3} \\
\Rightarrow \hspace{0.8in}
a_\circ^2 \  &\geq \ \left(1+\sqrt{2}\right)\ a_1 =\ r_1 \ a_{-1}\ a_{1}  \label{rr4}
\end{align}
%
%
%
%
\underline{\textbf{Odd Case}} \\
Being on the correct side of $\mathbf{\mathcal{H}_n}$, we have
\begin{align}
a_n\ &\geq \ \left(\frac{1+\sqrt{1+4r}}{2}\right)\ x^{1+d_1+\dots+d_{n-1}}   \notag \\
\Rightarrow \hspace{0.8in}
a_n^2\ &\geq \ \left(\frac{1+\sqrt{1+4r}}{2}\right)^2\ x^{2+2d_1+\dots+2d_{n-1}}\hspace{2in}   \notag    \\
&= \ \left(\frac{1+2r+\sqrt{1+4r}}{2}\right)\ x^{1+d_1+\dots+d_{n-1}}\  x^{1+d_1+\dots+d_{n-1}}  \notag \\
&= \ \left(\frac{1+2r+\sqrt{1+4r}}{2}\right)\  a_{n-1}\ a_{n+1} \notag
\intertext{above inequality is true for $r=1$}
\Rightarrow \hspace{0.8in}
a_n^2 \  &\geq \ \left(\frac{3+\sqrt{5}}{2}\right)\ a_{n-1}\ a_{n+1}\ =\ r_\circ \ a_{n-1}\ a_{n+1} \label{rr5} \\
\Rightarrow \hspace{0.8in}
a_n^2 \  &\geq \ \left(1+\sqrt{2}\right)\ a_{n-1}\ a_{n+1}\ =\ r_1 \ a_{n-1}\ a_{n+1}  \label{rr6}
\end{align}
$\underline{\textbf{Even Case}}$ \\
Being on the correct side of $\mathbf{\mathcal{H}_n}$ is equivalent to
\begin{align}
a_n\ &\geq \ \left(1+r\right)\ x^{1+d_1+\dots+d_{n-1}}
= \ \left(1+r\right)\ a_{n-1}  \hspace{2in} \label{rr7} \\
\Rightarrow \hspace{0.8in}
a_n\ &\geq \ \ 2 \ a_{n-1} \hspace{1.4in}  \label{rr8}
\end{align}
Since for $r=1$,  (\ref{rr1}), (\ref{rr3}), (\ref{rr5}) agrees with Lemma 3.1 (i) and (\ref{rr8}) with (ii) of McNamara and Sagan \cite{p2}. Thus any sequence in $\mathcal{R}$ is infinitely log-concave for $r=1$. Hence Generalized r-factor log-operator $\mathcal{L}_r$ and r-factor hypersurfaces agrees with the results obtained by \cite{p2} for $r=1.$
Also (\ref{rr2}), (\ref{rr4}), (\ref{rr6})and (\ref{rr7}) by Lemma \ref{revisedlemma3.1} proves theorem (\ref{maintheorem}) alternatively.
\end{proof}
\end{document}